\documentclass[11pt,a4paper]{article}
\usepackage{amsfonts,amsgen,amstext,amsbsy,amsopn,amsfonts,amssymb,amscd}
\usepackage{epsf,epsfig}
\usepackage{float}
\usepackage{ebezier,eepic}
\usepackage{color}
\usepackage{tikz}
\usepackage{multirow}
\usepackage{mathrsfs}
\usepackage{graphicx}
\usepackage{subfigure}
\usepackage[leqno]{amsmath}
\usepackage[amsmath,amsthm,thmmarks]{ntheorem}
\newtheorem{thm}{Theorem}[section]

\newtheorem{prop}[thm]{Proposition}
\newtheorem{prob}[thm]{Problem}
\newtheorem{example}[thm]{Example}
\newtheorem{lem}[thm]{Lemma}
\newtheorem{false statement}{False statement}
\newtheorem{cor}[thm]{Corollary}

\theoremstyle{definition}
\newtheorem{defn}[thm]{Definition}
\newtheorem{claim}{Claim}

\def\hh{\mathcal{H}}
\def\hf{\mathcal{F}}
\def\hg{\mathcal{G}}
\def\hht{\mathcal{T}}
\def\ha{\mathcal{A}}
\def\hb{\mathcal{B}}
\def\hs{\mathcal{S}}

\def\hi{\mathcal{I}}

\begin{document}
\title{\bf\Large Intersections and Distinct Intersections in Cross-intersecting Families}
\date{}
\author{Peter Frankl$^1$, Jian Wang$^2$\\[10pt]
$^{1}$R\'{e}nyi Institute, Budapest, Hungary\\[6pt]
$^{2}$Department of Mathematics\\
Taiyuan University of Technology\\
Taiyuan 030024, P. R. China\\[6pt]
E-mail:  $^1$peter.frankl@gmail.com, $^2$wangjian01@tyut.edu.cn
}

\maketitle

\begin{abstract}
Let $\hf,\hg$ be two cross-intersecting families of $k$-subsets of $\{1,2,\ldots,n\}$. Let $\hf\wedge \hg$, $\hi(\hf, \hg)$ denote the families of all intersections $F\cap G$ with $F\in \hf,G\in \hg$, and all distinct intersections $F\cap G$ with $F\neq G, F\in \hf,G\in \hg$, respectively. For a fixed $T\subset \{1,2,\ldots,n\}$, let $\hs_T$ be the family  of all $k$-subsets of $\{1,2,\ldots,n\}$ containing $T$. In the present paper, we show that $|\hf\wedge \hg|$ is maximized when $\hf=\hg=\hs_{\{1\}}$ for $n\geq 2k^2+8k$, while surprisingly  $|\hi(\hf, \hg)|$ is maximized when $\hf=\hs_{\{1,2\}}\cup \hs_{\{3,4\}}\cup \hs_{\{1,4,5\}}\cup \hs_{\{2,3,6\}}$ and $\hg=\hs_{\{1,3\}}\cup \hs_{\{2,4\}}\cup \hs_{\{1,4,6\}}\cup \hs_{\{2,3,5\}}$ for $n\geq 100k^2$. The maximum number of distinct intersections in a $t$-intersecting family is determined for $n\geq 3(t+2)^3k^2$ as well.

\end{abstract}
\section{Introduction}

Let $n,k$ be positive integers and let $[n]=\{1,2,\ldots,n\}$ denote the standard $n$-element set.  Let $\binom{[n]}{k}$ denote the collection of all $k$-subsets of $[n]$. Subsets of $\binom{[n]}{k}$ are called {\it $k$-uniform hypergraphs} or {\it $k$-graphs} for short.  A $k$-graph $\hf$ is called {\it intersecting} if $F\cap F'\neq \emptyset$ for all $F,F'\in \hf$. For a fixed set $T\subset [n]$, define the {\it $T$-star} $\hs_T$ by $\hs_T=\{S\in \binom{[n]}{k}\colon T\subset S\}$. We often write $\hs_{p}$, $\hs_{pq}$ and $\hs_{pqr}$ for $\hs_{\{p\}}$, $\hs_{\{p,q\}}$ and $\hs_{\{p,q,r\}}$, respectively. One of the most fundamental theorems in extremal set theory is the following:

\vspace{6pt}
{\noindent\bf Erd\H{o}s-Ko-Rado Theorem (\cite{ekr}).} Suppose that $n\geq 2k$ and $\hf\subset \binom{[n]}{k}$ is intersecting. Then
\begin{align}\label{ineq-ekr}
|\hf| \leq \binom{n-1}{k-1}.
\end{align}
\vspace{3pt}

Hilton and Milner \cite{HM67} proved  that $\hs_1$ is the only
family  that achieves equality in \eqref{ineq-ekr}  up to isomorphism for $n > 2k$.

Two families $\hf,\hg\subset \binom{[n]}{k}$ are called {\it cross-intersecting} if any two sets $F\in \hf, G\in \hg$ have non-empty intersection. If $\ha\subset \binom{[n]}{k}$ is intersecting, then $\hf=\ha$, $\hg=\ha$ are cross-intersecting. Therefore the following result is a strengthening of \eqref{ineq-ekr}.

\begin{thm}[\cite{Pyber86}]
Suppose that $n\geq 2k$ and $\hf,\hg\subset \binom{[n]}{k}$ are cross-intersecting. Then
\begin{align}\label{ineq-pyber}
|\hf||\hg|\leq \binom{n-1}{k-1}^2.
\end{align}
\end{thm}
Let us introduce the central notion of the present paper.
\begin{defn}
For $\hf,\hg\subset \binom{[n]}{k}$ define
\[
\hf\wedge \hg =\{F\cap G\colon F\in\hf, G\in\hg\}\ \mbox{ and } \hi(\hf, \hg) =\{F\cap G\colon F\in\hf, G\in\hg, F\neq G\}.
\]
Clearly $\hf\wedge \hg=(\hf\cap \hg)\cup \hi(\hf,\hg)$. For $\hf=\hg$, we often write $\hi(\hf)$ instead of $\hi(\hf,\hf)$.
\end{defn}

The first result of the present paper shows another extremal property of the full star.
\begin{thm}\label{main-1}
Suppose that $n\geq 2k^2+8k$, $\hf,\hg\subset \binom{[n]}{k}$ are cross-intersecting. Then
\begin{align}\label{ineq-1-1}
|\hf\wedge \hg|\leq \sum_{0\leq i\leq k-1}\binom{n-1}{i}
\end{align}
where equality holds if and only if $\hf=\hg=\hs_{1}$ up to isomorphism.
\end{thm}
\begin{cor}
Suppose that $n\geq 2k^2+8k$, $\hf\subset \binom{[n]}{k}$ is intersecting. Then
\begin{align*}
|\hf\wedge \hf|\leq \sum_{0\leq i\leq k-1}\binom{n-1}{i}
\end{align*}
where equality holds if and only if $\hf=\hs_{1}$ up to isomorphism.
\end{cor}

One would expect that both Theorem 1.3 and Corollary 1.4 hold for $n>ck$ for some absolute constant $c$. Unfortunately, we could not prove it. We can demonstrate the same results for $n>c'k^2/\log k$ with a more complicated proof.

Let us now consider the probably more natural quantity $|\hi(\hf,\hg)|$, namely the case that intersections of identical sets are not counted. Quite surprisingly the pairs of families maximizing $|\hi(\hf,\hg)|$ is rather peculiar. The fact that we can prove the optimality of such a pair shows the strength of our methods.

Let us define the two families
\[
\ha_1= \hs_{12}\cup \hs_{34}\cup \hs_{145}\cup \hs_{236}\mbox{ and } \ha_2=\hs_{13}\cup \hs_{24}\cup \hs_{146}\cup \hs_{235}.
\]
One can check that $\ha_1,\ha_2$ are cross-intersecting.
\begin{prop}
\begin{align}\label{eq-hihahb}
|\hi(\ha_1,\ha_2)|=&\quad 4\sum_{0\leq i\leq k-2}\binom{n-4}{i}+6\sum_{0\leq i\leq k-3}\binom{n-4}{i} +4\sum_{0\leq i\leq k-4}\binom{n-4}{i}\nonumber\\[5pt]
&\quad\quad\quad\quad+\sum_{0\leq i\leq k-5}\binom{n-4}{i}+2\sum_{i\leq i\leq k-3}\binom{n-6}{i}+\sum_{0\leq i\leq k-4}\binom{n-6}{i}.
\end{align}
\end{prop}
\begin{proof}
For any $A_1\in \ha_1$ and $A_2\in \ha_2$,
there are $\binom{4}{1}\sum\limits_{0\leq i\leq k-2}\binom{n-4}{i}$ distinct intersections for $|A_1\cap A_2\cap \{1,2,3,4\}|=1$. There are $\binom{4}{2}\sum\limits_{0\leq i\leq k-3}\binom{n-4}{i}$ distinct intersections for $|A_1\cap A_2\cap \{1,2,3,4\}|=2$. There are $\binom{4}{3}\sum\limits_{0\leq i\leq k-4}\binom{n-4}{i}$ distinct intersections for $|A_1\cap A_2\cap \{1,2,3,4\}|=3$. There are $\sum\limits_{0\leq i\leq k-5}\binom{n-4}{i}$ distinct intersections for $|A_1\cap A_2\cap \{1,2,3,4\}|=4$. There are $2\sum\limits_{0\leq i\leq k-3}\binom{n-6}{i}$ distinct intersections for $|A_1\cap A_2\cap \{1,2,3,4\}|=0$ and $|A_1\cap A_2\cap \{5,6\}|=1$. There are $\sum\limits_{0\leq i\leq k-4}\binom{n-6}{i}$ distinct intersections for $|A_1\cap A_2\cap \{1,2,3,4\}|=0$ and $|A_1\cap A_2\cap \{5,6\}|=2$. Thus the proposition follows.
\end{proof}

Our main result shows that $|\hi(\hf,\hg)|$ is maximized by $\ha_1,\ha_2$ over all cross-intersecting families $\hf,\hg\subset \binom{[n]}{k}$ for $n\geq 100k^2$.

\begin{thm}\label{main-2}
 If $\hf, \hg\subset \binom{[n]}{k}$ are cross-intersecting families and $n\geq 100k^2$, then $|\hi(\hf,\hg)|\leq |\hi(\ha_1,\ha_2)|$.
\end{thm}

Let $n\geq k>t$. A family $\hf\subset \binom{[n]}{k}$ is called $t$-intersecting if any two members of it intersect in at least $t$ elements. Note that for $n\leq 2k-t$ the whole set $\binom{[n]}{k}$ is $t$-intersecting. Thus we always assume that $n>2k-t$ when considering extremal problems for $t$-intersecting families.

Define
\[
\ha(n,k,t) =\left\{A\in \binom{[n]}{k}\colon |A\cap [t+2]|\geq t+1\right\}.
\]
This family was first defined in \cite{F78} and it is easily seen to be $t$-intersecting.

\begin{prop}
\begin{align}\label{eq-1-2}
|\hi(\ha(n,k,t))|= \binom{t+2}{t}&\sum_{0\leq i\leq k-t-1}\binom{n-t-2}{i}+\binom{t+2}{t+1}\sum_{0\leq i\leq k-t-2}\binom{n-t-2}{i}\nonumber\\[5pt]
&+\sum_{0\leq i\leq k-t-3}\binom{n-t-2}{i}.
\end{align}
\end{prop}
\begin{proof}
For any $A_1,A_2\in \ha(n,k,t)$, we have $|A_1\cap A_2\cap [t+2]|\geq t$. Note that $|A_i\cap [t+2]|\geq t+1$ for $i=1,2$. There are $\binom{t+2}{t}\sum\limits_{0\leq i\leq k-t-1}\binom{n-t-2}{i}$ distinct intersections for $|A_1\cap A_2\cap [t+2]|=t$. There are $\binom{t+2}{t+1}\sum\limits_{0\leq i\leq k-t-2}\binom{n-t-2}{i}$ distinct intersections for $|A_1\cap A_2\cap [t+2]|=t+1$. There are $\sum\limits_{0\leq i\leq k-t-2}\binom{n-t-2}{i}$ distinct intersections for $|A_1\cap A_2\cap [t+2]|=t+2$. Thus the proposition follows.
\end{proof}

Our third result shows that $|\hi(\hf)|$ is maximized by $\ha(n,k,t)$  over all intersecting families $\hf\subset\binom{[n]}{k}$ for $n\geq 3(t+2)^3k^2$.

\begin{thm}\label{main-3}
 If $\hf\subset \binom{[n]}{k}$ is a $t$-intersecting family and $n\geq 3(t+2)^3k^2$, then $\hi(\hf)\leq |\hi(\ha(n,k,t))|$.
\end{thm}
We should mention that this result was proved for the case $t=1$ in \cite{FKK2022}.

 Let us list some notions and results that we need for the  proofs.  Define the family of $t$-transversals of $\hf\subset\binom{[n]}{k}$:
\[
\hht_t(\hf) = \left\{T\subset [n]\colon |T|\leq k, |T\cap F|\geq t  \mbox{ for all } F\in \hf\right\}.
\]
Clearly,  if $\hf$ is $t$-intersecting then $\hf\subset \hht_t(\hf)$ and vice versa. The $t$-covering number $\tau_t(\hf)$ is defined as follows:
\[
\tau_t(\hf)=\min\{|T|\colon |T\cap F|\geq t \mbox{ for all }F\in \hf\}.
\]
For $t=1$, we often write $\hht(\hf), \tau(\hf)$ instead of $\hht_1(\hf),\tau_1(\hf)$, respectively.
If $\hf,\hg$ are cross-intersecting, then clearly $\hf\subset \hht(\hg)$ and $\hg\subset \hht(\hf)$.

Let us recall the following common notations:
$$\hf(i)=\{F\setminus\{i\}\colon i\in F\in \hf\}, \qquad \hf(\bar{i})= \{F\in\hf: i\notin F\}.$$
Note that $|\hf|=|\hf(i)|+|\hf(\bar{i})|$.

Define $\nu(\hf)$, the {\it matching number} of $\hf$ as the maximum number of pairwise disjoint edges in $\hf$. Note that $\nu(\hf)=1$ iff $\hf$ is intersecting. We need the following inequality generalising the Erd\H{o}s-Ko-Rado Theorem.

\begin{prop}[\cite{F87}]
Suppose that $\hf\subset \binom{[n]}{k}$ then
\begin{align}\label{ineq-EMCup}
|\hf|\leq \nu(\hf)\binom{n-1}{k-1}.
\end{align}
\end{prop}

An intersecting family $\hf$ is called {\it non-trivial} if $\cap_{F\in \hf} F = \emptyset$. We also need the following stability theorem concerning the Erd\H{o}s-Ko-Rado Theorem.

\vspace{6pt}
{\noindent\bf Hilton-Milner Theorem (\cite{HM67}).} If $n> 2k$ and $\hf\subset \binom{[n]}{k}$ is non-trivial intersecting, then
\begin{align}\label{ineq-nontrival}
|\hf| \leq \binom{n-1}{k-1}-\binom{n-k-1}{k-1}+1.
\end{align}

Let us list some inequalities that will be used frequently in the proof.
\begin{prop}
Let $n,k,\ell,t,p$ be positive integers with $k> \ell$, $k> t$ and $n> 2k+p$. Then
\begin{align}
&\binom{n}{k}\leq \frac{n-p}{n-p (k+1)}\binom{n-p}{k},\label{ineq-4-0}\\[5pt]
&\sum_{0\leq i\leq k-\ell} \binom{n-t}{i} \leq \frac{n-t-p}{n-t-pk}\sum_{0\leq i\leq k-\ell} \binom{n-t-p}{i}, \label{ineq-4-1}\\[5pt]
&\sum_{0\leq i\leq k-\ell-1} \binom{n-t}{i} \leq \frac{k}{n-t-k}\sum_{0\leq i\leq k-\ell} \binom{n-t}{i},\label{ineq-4-2}\\[5pt]
& \mbox{ for }\ell\geq t+1,\ \sum_{t\leq j\leq \ell}\binom{\ell}{j}\geq \frac{1}{2t+2}\sum_{t\leq j\leq \ell+1}\binom{\ell+1}{j}.\label{ineq-4-6}
\end{align}
\end{prop}
\begin{proof}
Note that
\[
\frac{\binom{n-p}{k}}{\binom{n}{k}} =\frac{(n-k)(n-k-1)\cdots (n-k-p+1)}{n(n-1)\cdots (n-p+1)}\geq \left(1-\frac{k}{n-p}\right)^p \geq 1-\frac{p k}{n-p}.
\]
Then \eqref{ineq-4-0} holds. By \eqref{ineq-4-0}, we have for $i<k$
\[
\binom{n-t}{i}\leq \frac{n-t-p}{n-t-p (i+1)}\binom{n-t-p}{i}\leq \frac{n-t-p}{n-t-pk}\binom{n-t-p}{i},
\]
and thereby \eqref{ineq-4-1} follows. Since
\[
\binom{n-t}{i-1}/\binom{n-t}{i}=\frac{i}{n-t-i+1} \leq \frac{k}{n-t-k},
\]
 we obtain \eqref{ineq-4-2}.

For $\ell \geq 2t$, since
 \[
 \sum_{t\leq j\leq \ell}\binom{\ell}{j}\geq 2^{\ell-1} \mbox{ and }\sum_{t\leq j\leq\ell+1}\binom{\ell+1}{j}\leq 2^{\ell+1},
 \]
we see that
 \[
 \frac{\sum\limits_{t\leq j\leq\ell}\binom{\ell}{j}}{\sum\limits_{t\leq j\leq\ell+1}\binom{\ell+1}{j}}  \geq \frac{1}{4}.
 \]
 For $t+1\leq \ell \leq 2t$,
 \begin{align*}
  \sum_{t\leq j\leq \ell}\binom{\ell}{j}&\geq  \sum_{t\leq j\leq\ell}\frac{\ell+1-j}{\ell+1}\binom{\ell+1}{j}\\
  &\geq \sum_{t\leq j\leq\ell-1}\frac{\ell+1-j}{\ell+1}\binom{\ell+1}{j}+\frac{1}{\ell+1}\binom{\ell+1}{\ell}\\
  &\geq \frac{1}{\ell+1}\sum_{t\leq j\leq\ell-1}\binom{\ell+1}{j}+\frac{1}{\ell+2}\left(\binom{\ell+1}{\ell}+\binom{\ell+1}{\ell+1}\right)\\
  &> \frac{1}{2t+2}\sum_{t\leq j\leq\ell+1}\binom{\ell+1}{j}.
 \end{align*}
 Thus \eqref{ineq-4-6} holds.
\end{proof}

\section{Intersections in cross-intersecting families}

In this section, we determine the maximum size of $\hf\wedge \hg$ over all cross-intersecting families $\hf,\hg\subset \binom{[n]}{k}$. We also determine the maximum size of $(\hf_1\wedge \hg_1)\cup(\hf_2\wedge \hg_2)$ over all families $\hf_1,\hf_2,\hg_1,\hg_2\subset \binom{[n]}{k}$ with $\hf_1,\hg_1$ being cross-intersecting and $\hf_2,\hg_2$ being cross-intersecting. This result will be used in Section 3.

First we prove a key proposition to the proof of Theorem \ref{main-1}.

\begin{prop}\label{prop2-1}
Let $\hf,\hg\subset \binom{[n]}{k}$ be cross-intersecting families and set $\hh=\hi(\hf,\hg)\cap \binom{[n]}{k-1}$. Then $\nu(\hh)\leq 4$.
\end{prop}

\begin{proof}
Suppose that $F_i\cap G_i=D_i$ are pairwise disjoint $(k-1)$-sets, $0\leq i\leq 4$. Define $x_i,y_i$ by $F_i=D_i\cup \{x_i\}, G_i=D_i\cup \{y_i\}$ and note that $x_i\neq y_i$. There are altogether $5\times 4$ conditions $F_i\cap G_j \neq \emptyset$ to satisfy. Each of them is assured by either of the following three relations: $x_i\in D_j,\ y_j\in D_i, x_i=y_j$. From the first two types there are at most one for each $x_i$ and $y_j$. Altogether at most $5+5=10$. If no multiple equalities (e.g. $x_1=y_2=y_3$) exist, we get only at most 5 more relations and $10+5<20$. Thus there must be places of coincidence, say by symmetry that of the form $x_i=x_{i'}$. Thus, again by symmetry, we may assume that $x_i\notin D_0$ for $0\leq i\leq 4$. Note that $y_0\in D_i$ holds for at most one value of $i$. Without loss of generality assume $y_0\notin D_i$, $1\leq i\leq 3$. By $F_i\cap G_0\neq \emptyset$, $y_0=x_i$, $i=1,2,3$. Look at $y_1$. By symmetry assume $y_1\notin D_2$. Now $G_1\cap F_2\neq \emptyset$ implies $y_1=x_2$. Hence $y_1=x_1$, a contradiction.
\end{proof}

Let $D_1,D_2,D_3,D_4$ be pairwise disjoint $(k-1)$-sets. Pick an element $d_i\in D_i$, $i=1,2,3$. Define $x_i,y_i$ by $x_1=x_2=y_4=d_3$, $x_3=y_1=d_2$ and  $x_4=y_2=y_3=d_1$. Setting $F_i=D_i\cup \{x_i\}$, $G_i=D_j\cup\{y_j\}$. One can check easily that $F_i\cap G_j\neq\emptyset$ for $1\leq i\neq j\leq 4$. This example shows that Proposition \ref{prop2-1} is best possible.

\begin{proof}[Proof of Theorem \ref{main-1}]
We distinguish two cases. First we suppose that
\begin{align}\label{ineq-2-1}
|\hf\cap \hg|> \binom{n-1}{k-1}-\binom{n-k-1}{k-1}+1.
\end{align}
Since $\hf,\hg$ are cross-intersecting, $\hf\cap \hg$ is intersecting. By \eqref{ineq-nontrival} and \eqref{ineq-2-1}, without loss of generality, we assume that $1\in F$ for all $F\in \hf\cap \hg$. We claim that $1\in H$ for all $H\in \hf\cup \hg$. Indeed, if $1\notin H\in \hf\cup \hg$ then $H\cap F\neq \emptyset$ for $F\in \hf\cap \hg$ yields
\[
|\hf\cap \hg|\leq \binom{n-1}{k-1}-\binom{n-k-1}{k-1}
\]
contradicting \eqref{ineq-2-1}. We proved that $1\in H$ for all $H\in \hf\cup \hg$ and thereby \eqref{ineq-1-1} holds.

Suppose next that \eqref{ineq-2-1} does not hold. By Proposition \ref{prop2-1} and  \eqref{ineq-EMCup}, we have for $n\geq 5k$,
\[
\left|\hi(\hf,\hg)\cap \binom{[n]}{k-1}\right| \leq  4\binom{n-1}{k-2}\overset{\eqref{ineq-4-0}}{\leq}\frac{4 (n-2)}{n-k}\binom{n-2}{k-2} \leq 5\binom{n-2}{k-2}.
\]
Since the remaining sets in $\hi(\hf,\hg)$ are  of size  at most $k-2$, we have
\[
\left|\hi(\hf,\hg)\right| \leq 5\binom{n-2}{k-2} +\sum_{0\leq i\leq k-2} \binom{n}{i}.
\]
Moreover,
\[
|\hf\cap \hg|\leq  \binom{n-1}{k-1}-\binom{n-k-1}{k-1}+1\leq k\binom{n-2}{k-2}.
\]
Thus, for $n\geq 2k+1$ we have
\begin{align*}
|\hf\wedge \hg|&\leq (k+5)\binom{n-2}{k-2}+\sum_{0\leq i\leq k-2} \binom{n}{i} \\
&\overset{\eqref{ineq-4-2}}{\leq} \frac{(k+5)(k-1)}{n-1}\binom{n-1}{k-1} +\frac{k}{n-k}\sum_{0\leq i\leq k-1} \binom{n}{i}\\
&\overset{\eqref{ineq-4-1}}{\leq} \frac{(k+5)(k-1)}{n-1}\binom{n-1}{k-1} +\frac{k}{n-k}\frac{n-1}{n-k}\sum_{0\leq i\leq k-1} \binom{n-1}{i}.
\end{align*}
Note that $n\geq 2k^2+8k$ implies
\[
\frac{(k+5)(k-1)}{n-1} \leq \frac{1}{2}
\]
and
\[
\frac{k}{n-k}\frac{n-1}{n-k}< \frac{k}{n-k}\left(1+\frac{k}{n-k}\right)<\frac{k}{2k^2+7k}\left(1+\frac{k}{2k^2}\right)=\frac{2k+1}{2k(2k+7)}\leq
\frac{1}{2}.
\]
Thus,
\[
|\hf\wedge \hg| \leq \frac{1}{2}\binom{n-1}{k-1} +\frac{1}{2}\sum_{0\leq i\leq k-1} \binom{n-1}{i} < \sum_{0\leq i\leq k-1}\binom{n-1}{i}.
\]
\end{proof}

\begin{lem}\label{lem2-2}
Suppose that $n\geq 2k^2+9k$, $\hf_1,\hg_1\subset \binom{[n]}{k}$ are cross-intersecting and  $\hf_2,\hg_2\subset \binom{[n]}{k}$ are cross-intersecting. Then
\begin{align}\label{ineq-2-3}
|(\hf_1\wedge \hg_1)\cup (\hf_2\wedge \hg_2)|\leq 2\sum_{0\leq i\leq k-1}\binom{n-2}{i}+\sum_{0\leq i\leq k-2}\binom{n-2}{i}.
\end{align}
with equality holding if and only if $\hf_1=\hg_1=\hs_{1}$ and $\hf_2=\hg_2=\hs_{2}$ up to isomorphism.
\end{lem}

\begin{proof}
By Theorem \ref{main-1}, for $j=1,2$
\[
|\hf_j\wedge \hg_j|\leq \sum_{0\leq i\leq k-1}\binom{n-1}{i}.
\]
By Proposition \ref{prop2-1} and \eqref{ineq-EMCup}, for $j=1,2$
\[
\left|\hi(\hf_j,\hg_j)\cap \binom{[n]}{k-1}\right| \leq 4\binom{n-1}{k-2}\overset{\eqref{ineq-4-0}}{\leq}\frac{4 (n-2)}{n-k}\binom{n-2}{k-2} \leq 5\binom{n-2}{k-2}.
\]
Since the remaining sets in $\hi(\hf_j,\hg_j)$ are  of size at most $k-2$, for $j=1,2$
\[
\left|\hi(\hf_j,\hg_j)\right| \leq 5\binom{n-2}{k-2} +\sum_{0\leq i\leq k-2} \binom{n}{i}.
\]
If $|\hf_j\cap \hg_j|\leq \binom{n-1}{k-1}-\binom{n-k-1}{k-1}\leq k\binom{n-2}{k-2}$ for some $j\in \{1,2\}$, then for $n\geq 2k+2$
\begin{align*}
&\quad |(\hf_1\wedge \hg_1)\cup (\hf_2\wedge \hg_2)| \\[5pt]
&\leq  |\hf_1\wedge \hg_1|+|\hf_2\wedge \hg_2|\\[5pt]
&\leq k\binom{n-2}{k-2}+5\binom{n-2}{k-2} +\sum_{0\leq i\leq k-2} \binom{n}{i}+ \sum_{0\leq i\leq k-1}\binom{n-1}{i}\\[5pt]
&\overset{\eqref{ineq-4-1}}{\leq } \frac{(k+5)(k-1)}{n-k}\binom{n-2}{k-1}+\frac{n-2}{n-2k}\sum_{0\leq i\leq k-2} \binom{n-2}{i}+\frac{n-2}{n-1-k}\sum_{0\leq i\leq k-1} \binom{n-2}{i}.
\end{align*}
Note that $n\geq 2k^2+9k\geq 10k$ implies
\[
\frac{(k+5)(k-1)}{n-k} \leq \frac{1}{2},\ \frac{n-2}{n-2k} \leq \frac{5}{4}\mbox{ and } \frac{n-2}{n-1-k}\leq \frac{5}{4}.
\]
Thus,
\begin{align*}
&\quad |(\hf_1\wedge \hg_1)\cup (\hf_2\wedge \hg_2)| \\[5pt]
&\leq \frac{1}{2}\binom{n-2}{k-1}+\frac{5}{4}\sum_{0\leq i\leq k-2} \binom{n-2}{i}+\frac{5}{4}\sum_{0\leq i\leq k-1} \binom{n-2}{i}\\[5pt]
&< 2\sum_{0\leq i\leq k-1}\binom{n-2}{i}+\sum_{0\leq i\leq k-2}\binom{n-2}{i}.
\end{align*}

Thus we may assume that $|\hf_j\cap \hg_j|\geq  \binom{n-1}{k-1}-\binom{n-k-1}{k-1}+1$ for each $j=1,2$. By \eqref{ineq-nontrival},   both $\hf_1\cap \hg_1$ and $\hf_2\cap \hg_2$ are trivial intersecting families. By the same argument as in Theorem \ref{main-1}, we see that there exist $x,y$ such that $\hf_1\cup \hg_1\subset \hs_{x}$ and $\hf_2\cup \hg_2\subset \hs_{y}$. If $x\neq y$, then we are done. If $x=y$, then
\begin{align*}
|(\hf_1\wedge \hg_1)\cup (\hf_2\wedge \hg_2)|&\leq |\hs_{x}\wedge \hs_{x}|\\[5pt]
 &=\sum_{0\leq i\leq k-1}\binom{n-1}{i}\\[5pt]
 &<2\sum_{0\leq i\leq k-1}\binom{n-2}{i}+\sum_{0\leq i\leq k-2}\binom{n-2}{i}.
\end{align*}
\end{proof}

\section{Distinct intersections in cross-intersecting families}

In this section, we determine the maximum number of distinct intersections in cross-intersecting families.

For the proof, we need the following  notion of  basis. Two cross-intersecting families $\hf, \hg$ are called {\it saturated} if any cross-intersecting families $\tilde{\hf},\tilde{\hg}$ with $\hf\subset \tilde{\hf}$, $\hg\subset \tilde{\hg}$ have $\hf= \tilde{\hf}$ and $\hg= \tilde{\hg}$. Since $\hf\subset \tilde{\hf}$ and  $\hg\subset \tilde{\hg}$ imply $\hi(\hf,\hg)\subset \hi(\tilde{\hf},\tilde{\hg})$, we may always assume that $\hf, \hg$ are saturated when maximizing the size of $\hi(\hf,\hg)$. Let $\hb(\hf)$ be the family of minimal (for containment) sets in $\hht(\hg)$ and let $\hb(\hg)$ be the family of minimal sets in $\hht(\hf)$. Let us prove some properties of the basis.

\begin{lem}\label{lem3-1}
Suppose that $\hf,\hg\subset \binom{[n]}{k}$ are saturated cross-intersecting families. Then (i) and (ii) hold.
\begin{itemize}
  \item[(i)] Both $\hb(\hf)$ and $\hb(\hg)$ are antichains, and $\hb(\hf), \hb(\hg)$ are cross-intersecting,
  \item[(ii)] $\hf=\left\{F\in \binom{[n]}{k}\colon \exists B\in \hb(\hf), B\subset F\right\}$ and $\hg=\left\{G\in \binom{[n]}{k}\colon \exists B\in \hb(\hg), B\subset G\right\}$.
\end{itemize}
\end{lem}

\begin{proof}
 (i) Clearly,  $\hb(\hf)$ and $\hb(\hg)$ are both anti-chains. Suppose for contradiction that $B\in \hb(\hf), B'\in \hb(\hg)$ but $B\cap B'=\emptyset$. If $|B|=|B'|=k$, then $B\in \hf,B'\in \hg$ follows from saturatedness, a contradiction. If $|B|<k$, then there exists $F\supset B$ such that $|F|=k$ and $|F\cap B'|=|B\cap B'|=0$. By definition $F\in \hht(\hg)$. Since $\hf,\hg$ are saturated, we see that $F\in \hf$. But this contradicts the assumption that $B'$ is a transversal of $\hf$. Since $\hf,\hg$ are saturated, (ii) is immediate from the definition of $\hb(\hf)$ and  $\hb(\hg)$.
\end{proof}

Let $r(\hb)=\max\{|B|\colon B\in \hb\}$ and $s(\hb)=\min\{|B|\colon B\in \hb\}$. For any $\ell$ with $s(\hb)\leq \ell \leq r(\hb)$, define
\[
\hb^{(\ell)} = \left\{B\in \hb\colon |B|=\ell\right\} \mbox{ and } \hb^{(\leq \ell)} = \bigcup_{i=s(\hb)}^\ell\hb^{(i)}.
\]
It is easy to see that $s(\hb(\hg))=\tau(\hf)$.

By  a branching process, we establish an upper bound on the size of the basis.

\begin{lem}\label{lem3-2}
Suppose that $\hf,\hg\subset \binom{[n]}{k}$ are saturated cross-intersecting families. Let $\hb_1=\hb(\hf)$ and $\hb_2=\hb(\hg)$. For each $i=1,2$, if $s(\hb_i)\geq 2$ and $\tau(\hb_i^{(\leq r_i)})\geq 2$  then
\begin{align}\label{ineq-crosshb}
\sum_{r_i\leq \ell\leq k}\ell^{-2} k^{-\ell+2}|\hb_{3-i}^{(\ell)}|\leq 1.
\end{align}
\end{lem}

\begin{proof}
By symmetry, it is sufficient to prove the lemma only for $i=1$. For the proof we use a branching process. During the proof {\it a sequence} $S=(x_1,x_2,\ldots,x_\ell)$ is an ordered sequence of distinct elements of $[n]$ and we use $\widehat{S}$ to denote the underlying unordered set $\{x_1,x_2,\ldots,x_\ell\}$. At the beginning, we assign weight 1 to the empty sequence $S_{\emptyset}$. At the first stage, we choose $B_{1,1}\in \hb_1$ with $|B_{1,1}|=s(\hb_1)$. For any vertex $x_1\in B_{1,1}$, define one sequence $(x_1)$ and assign the weight $s(\hb_1)^{-1}$ to it.

At the second stage, since $\tau(\hb_1^{(\leq r_1)})\geq 2$, for each sequence $S=(x_1)$ we may choose $B_{1,2}\in \hb_1^{(\leq r_1)}$ such that $x_1\notin B_{1,2}$. Then we replace $S=(x_1)$ by $|B_{1,2}|$ sequences of the form $(x_1,y)$ with $y\in B_{1,2}$ and weight $\frac{w(S)}{|B_{1,2}|}$.

In each subsequent stage, we pick a sequence $S=(x_1,\ldots,x_p)$ and denote its weight by $w(S)$. If $\widehat{S}\cap B_1\neq \emptyset$ holds for all $B_1\in \hb_1$ then we do nothing. Otherwise we pick $B_1\in \hb_1$ satisfying $\widehat{S}\cap B_1= \emptyset$ and replace $S$ by the $|B_1|$ sequences $(x_1,\ldots,x_p,y)$ with $y\in B_1$ and assign weight $\frac{w(S)}{|B_1|}$ to each of them. Clearly, the total weight is always 1.

We continue until $\widehat{S}\cap B_1\neq \emptyset$ for all sequences and all $B_1\in \hb_1$. Since $[n]$ is finite, each sequence has length at most $n$ and  eventually the process stops. Let $\hs$ be the collection of sequences that survived in the end of the branching process and let $\hs^{(\ell)}$ be the collection of sequences in $\hs$ with length $\ell$.

\begin{claim}
To each $B_2\in \hb_{2}^{(\ell)}$ with $\ell\geq r_1$ there is some sequence $S\in \hs^{(\ell)}$ with $\widehat{S}=B_2$.
\end{claim}
\begin{proof}
Let us suppose the contrary and let $S=(x_1,\ldots,x_p)$ be a sequence of maximal length that occurred
at some stage of the branching process satisfying $\widehat{S}\subsetneqq B_2$. Since $\hb_1,\hb_2$ are cross-intersecting,
 $B_{1,1}\cap B_2\neq  \emptyset$, implying that $p\geq 1$. Since $\widehat{S}$ is a proper subset of $B_2$ and $B_2\in \hb_2=\hb(\hg)$, it follows that $\widehat{S}\notin \hb(\hg)\subset\hht(\hf)$. Thereby there exists $F\in \hf$ with $\widehat{S} \cap F= \emptyset$.  In view of Lemma \ref{lem3-1} (ii), we can find  $B_1'\in \hb_1$ such that $\widehat{S} \cap B_1'=  \emptyset$. Thus at some point we picked $S$ and some $\tilde{B}_1\in \hb_1$ with $\widehat{S} \cap \tilde{B}_1=  \emptyset$. Since $\hb_1,\hb_2$ are cross-intersecting, $B_2\cap \tilde{B}_1\neq  \emptyset$. Consequently, for each $y\in B_2\cap \tilde{B}_1$ the sequence $(x_1,\ldots,x_p,y)$ occurred in the branching process. This contradicts the maximality of $p$. Hence there is an $S$ at some stage satisfying $\widehat{S}= B_2$. Since $\hb_1,\hb_2$ are cross-intersecting, $\widehat{S}\cap B_1'=B_2\cap B_1'\neq \emptyset$ for all $B_1'\in \hb_1$. Thus $\widehat{S}\in \hs$ and the claim holds.
\end{proof}
By Claim 1, we see that $|\hb_2^{(\ell)}|\leq |\hs^{(\ell)}|$ for all $\ell\geq r_1$.  Let $S=(x_1,\ldots,x_\ell)\in \hs^{(\ell)}$ and let $S_i=(x_1,\ldots,x_i)$ for $i=1,\ldots,\ell$. At the first stage, $w(S_1)=1/s(\hb_1)$.   Assume that $B_{1,i}$ is the selected set when replacing $S_{i-1}$ in the  branching process for $i=2,\ldots,\ell$. Clearly,  $x_i\in B_i$, $B_{1,2}\in \hb_1^{(\leq r_1)}$ and
\[
w(S)= \frac{1}{s(\hb_1)}\prod_{i=2}^\ell \frac{1}{|B_{1,i}|}.
\]
Note that $s(\hb_1)\leq \ell$, $|B_{1,2}|= r_1\leq \ell$ and $|B_{1,i}|\leq k$ for $i\geq 3$. It follows that
\[
w(S)\geq \left(\ell^2 k^{\ell-2}\right)^{-1}=\ell^{-2} k^{-\ell+2}.
\]
Thus,
\[
\sum_{r_1\leq \ell\leq k}\ell^{-2} k^{-\ell+2}|\hb_{2}^{(\ell)}|\leq \sum_{r_1\leq \ell\leq k}\sum_{S\in \hs^{(\ell)}}w(S)\leq \sum_{S\in \hs}w(S)=1.
\]
\end{proof}

For the proof of Theorem \ref{main-2}, we also need the following lemma.
\begin{lem}\label{lem3-3}
Suppose that  $\hf\subset \binom{[n]}{k}, \hg\subset \binom{[n]}{k-1}$ are cross-intersecting. Then
\[
|\hi(\hf,\hg)|\leq  2\binom{n-1}{k-2}+(2k+1)\binom{n-1}{k-3}+\sum_{0\leq i\leq k-3}\binom{n}{i}.
\]
\end{lem}

\begin{proof}
Let $\hh_1=\hi(\hf,\hg)\cap \binom{[n]}{k-1}$. We claim that $\nu(\hh_1)\leq 2$. Otherwise, let $G_i=F_i\cap G_i$, $i=1,2,3$, be three pairwise disjoint members in $\hh_1$ with $F_i\in \hf$, $G_i\in \hg$. Define $x_i$ by  $F_i\setminus G_i = \{x_i\}$. By symmetry we may assume that $x_1\notin G_3$. Then $F_1,G_3$ are disjoint, contradicting the fact that $\hf,\hg$ are cross-intersecting. Thus $\nu(\hh_1)\leq 2$.

If $\nu(\hh_1)\leq 1$, then \eqref{ineq-ekr} implies $|\hh_1|\leq \binom{n-1}{k-2}$. Since the remaining sets in $\hi(\hf,\hg)$ are all of size at most $k-2$, it follows that
\[
|\hi(\hf,\hg)|\leq \binom{n-1}{k-2}+\sum_{0\leq i\leq k-2}\binom{n}{i}.
\]
If $\nu(\hh_1)= 2$, let $G_1=F_1\cap G_1, G_2=F_2\cap G_2$ be two disjoint members in $\hh_1$ and let $\hh_2=\hi(\hf,\hg)\cap \binom{[n]\setminus (F_1\cup F_2)}{k-2}$.  We claim that $\hh_2$ is intersecting. Suppose not, let $D_3=F_3\cap G_3, D_4=F_4\cap G_4$ be two disjoint members in $\hh_2$. Define $x_i$ by  $F_i\setminus G_i= \{x_i\}$ for $i=1,2$ and define $x_i,y_i,z_i$ by  $F_i\setminus D_i= \{x_i,y_i\}$, $G_i\setminus D_i= \{z_i\}$  for $i=3,4$. Since $F_3\cap G_1\neq \emptyset$ and $F_3\cap G_2\neq \emptyset$, by symmetry we may assume that $x_3\in G_1$ and $y_3\in G_2$. Similarly,  assume that $x_4\in G_1$ and $y_4\in G_2$. Since $F_1\cap G_3\neq \emptyset$ and $F_2\cap G_3\neq \emptyset$, we see that $z_3\in F_1\cap F_2$. It follows that $x_1=x_2=z_3$. Similarly we have $x_1=x_2=z_4$. But then $F_3,G_4$ are disjoint, contradicting the fact that $\hf,\hg$ are cross-intersecting. Thus $\hh_2$ is intersecting. By \eqref{ineq-ekr} we have
\[
\left|\hi(\hf,\hg)\cap \binom{[n]}{k-2}\right|\leq |F_1\cup F_2|\binom{n-1}{k-3}+\binom{n-2k}{k-3}\leq  (2k+1)\binom{n-1}{k-3}.
\]
By \eqref{ineq-EMCup} we obtain that
\[
\left|\hi(\hf,\hg)\cap \binom{[n]}{k-1}\right|\leq  2\binom{n-1}{k-2}.
\]
Hence
\[
|\hi(\hf,\hg)|\leq 2\binom{n-1}{k-2}+(2k+1)\binom{n-1}{k-3}+\sum_{0\leq i\leq k-3}\binom{n}{i}.
\]
\end{proof}
\begin{cor}\label{cor3-4}
Let $\hf, \hg\subset \binom{[n]}{k}$ be cross-intersecting families. If $\hg$ is a star,  then
\begin{align}\label{ineq-cor3-4}
\hi(\hf,\hg)\leq 2\sum_{0\leq i\leq k-2}\binom{n-1}{i}+ \binom{n-2}{k-2}+(2k+1)\binom{n-2}{k-3}.
\end{align}
\end{cor}
\begin{proof}
Assume without loss of generality that $\hf$ and $\hg$ are saturated. Since $\hg$ is a star, we may assume that $\hg\subset \hs_{1}$. Then $\{1\}\in \hht(\hg)$ whence $\{1\}\in \hb(\hf)$. By Lemma \ref{lem3-1} (ii) $\hs_{1} \subset \hf$. Note that $\hf(\bar{1})\subset \binom{[2,n]}{k}, \hg(1)\subset \binom{[2,n]}{k-1}$ are cross-intersecting. By Lemma \ref{lem3-3}, we have
\begin{align*}
|\hi(\hf(\bar{1}),\hg(1))| &\leq 2\binom{n-2}{k-2}+(2k+1)\binom{n-2}{k-3}+\sum_{0\leq i\leq k-3}\binom{n-1}{i}\\[5pt]
&\leq  \sum_{0\leq i\leq k-2}\binom{n-1}{i}+\binom{n-2}{k-2}+(2k+1)\binom{n-2}{k-3}.
\end{align*}
Thus,
\begin{align*}
|\hi(\hf,\hg)|&\leq |\hi(\hs_1,\hs_1)|+ |\hi(\hf(\bar{1}),\hg(1))|\\[5pt]
&\leq 2\sum_{0\leq i\leq k-2}\binom{n-1}{i}+\binom{n-2}{k-2}+(2k+1)\binom{n-2}{k-3}.
\end{align*}
\end{proof}

Now we are in position to prove the main theorem.

\begin{proof}[Proof of Theorem \ref{main-2}]
Let $\hb_1=\hb(\hf),\hb_2=\hb(\hg)$ and  let $s_1=s(\hb_1)$, $s_2=s(\hb_2)$.  Suppose first that $\min\{s_1,s_2\}=1$. By symmetry let $s_2=1$, then $\hg$ is a star. By \eqref{ineq-cor3-4} and $n\geq 2k+3$, we have
\begin{align*}\label{ineq-3-7}
&\quad|\hi(\hf,\hg)|\\[5pt]
&\leq 2\sum_{0\leq i\leq k-2}\binom{n-1}{i}+\binom{n-2}{k-2}+(2k+1)\binom{n-2}{k-3}\\[5pt]
&\overset{\eqref{ineq-4-0}}{\leq} 2\sum_{0\leq i\leq k-2}\binom{n-1}{i}+\frac{n-2-2}{n-2-2(k-1)}\binom{n-4}{k-2}+\frac{(2k+1)(k-2)}{n-1}\binom{n-1}{k-2}\\[5pt]
&\overset{\eqref{ineq-4-0}}{\leq} 2\sum_{0\leq i\leq k-2}\binom{n-1}{i}+\frac{n-4}{n-2k}\binom{n-4}{k-2}+\frac{(2k+1)(k-2)}{n-1}\frac{n-1-3}{n-1-3(k-1)}\binom{n-4}{k-2}\\[5pt]
&\overset{\eqref{ineq-4-1}}{\leq} \frac{2 (n-4)}{n-1-3k}\sum_{0\leq i\leq k-2}\binom{n-4}{i}+\frac{n-4}{n-2k}\binom{n-4}{k-2}+\frac{(2k+1)(k-2)(n-4)}{(n-1)(n-3k)}\binom{n-2}{k-3}.
\end{align*}
Note that $n\geq 63 k$ implies
\[
\frac{2(n-4)}{n-1-3k}\leq \frac{21}{10},\ \frac{n-4}{n-2k}\leq \frac{11}{10}
\]
and $n\geq 44 k^2$ implies
\[
\frac{(2k+1)(k-2)}{n-1}\leq \frac{1}{22},\  \frac{n-4}{n-3k} \leq \frac{11}{10}.
\]
Thus,
\begin{align}\label{ineq-3-9}
|\hi(\hf,\hg)|&\leq \frac{21}{10}\sum_{0\leq i\leq k-2}\binom{n-4}{i}+\frac{11}{10}\binom{n-4}{k-2}+\frac{1}{20}\binom{n-4}{k-2}\nonumber\\[5pt]
&<  \frac{13}{4}\sum_{0\leq i\leq k-2}\binom{n-4}{i}<|\hi(\ha_1,\ha_2)|.
\end{align}

Thus, we may assume that $s_1, s_2\geq 2$.
 Let us partition $\hf$ into $\hf^{(s_1)}\cup \ldots\cup\hf^{(k)}$ where $F\in \hf^{(\ell)}$ if $\max\{|B|\colon B\in \hb_1, B\subset F\}=\ell$. Similarly,  partition $\hg$ into $\hg^{(s_2)}\cup \ldots\cup\hg^{(k)}$ where $G\in \hg^{(\ell)}$ if $\max\{|B|\colon B\in \hb_2, B\subset G\}=\ell$.

Fix an $F\in \hf^{(\ell)}$ with $B_1\subset F, B_1\in \hb_1^{(\ell)}$. For an arbitrary $G\in \hg$, we have
 \[
 F\cap G =(B_1\cap G)\cup ((F\setminus B_1)\cap G),
 \]
 where $B_1\cap G\neq \emptyset$ and $|(F\setminus B_1)\cap G| \leq |F\setminus B_1|=k-\ell$. It follows that for $s_1\leq \ell\leq k$
 \begin{align}\label{ineq-3-3}
 |\hi(\hf^{(\ell)},\hg)| \leq |\hb_1^{(\ell)}|(2^\ell-1)\sum_{0\leq i\leq k-\ell }\binom{n-1}{i}.
 \end{align}
 Similarly, for $s_2\leq \ell\leq k$
 \begin{align}\label{ineq-3-4}
 |\hi(\hf,\hg^{(\ell)})| \leq |\hb_2^{(\ell)}|(2^\ell-1)\sum_{0\leq i\leq k-\ell }\binom{n-1}{i}.
 \end{align}

 Let $\alpha$ be the smallest integer such that $\tau(\hb_1^{(\leq \alpha)})\geq 2$ and let $\beta$ be the smallest integer such that $\tau(\hb_2^{(\leq \beta)})\geq 2$. By symmetry, we may assume that $\alpha \geq \beta$. We distinguish three cases.

{\bf Case 1.} $\beta\geq 3$.
Let $\hf'=\hf^{(s_1)}\cup \cdots \cup \hf^{(\beta-1)}$.  Note that $\hf'$ and $\hg$ are cross-intersecting and $\hf'$ is a star. By \eqref{ineq-cor3-4} and \eqref{ineq-3-9}, we have
\begin{align}\label{ineq-3-5}
 |\hi(\hf',\hg)|\leq 2\sum_{0\leq i\leq k-2}\binom{n-1}{i}+\binom{n-2}{k-2}+(2k+1)\binom{n-2}{k-3}<\frac{13}{4}\sum_{0\leq i\leq k-2}\binom{n-4}{i}.
\end{align}
Define
\[
f(n,k,\ell) =  2^\ell \ell^2k^{\ell-2} \sum_{0\leq i\leq k-\ell}\binom{n-1}{i}.
\]
and let
\[
\lambda_{\ell} = \ell^{-2} k^{-\ell+2}|\hb_{1}^{(\ell)}|.
\]
By \eqref{ineq-3-3}, we see
\[
 \sum_{\beta\leq \ell\leq k} |\hi(\hf^{(\ell)},\hg)|= \sum_{\beta\leq \ell\leq k} \lambda_{\ell} f(n,k,\ell).
\]
Since
 \[
 \frac{f(n,k,\ell)}{f(n,k,\ell+1)}=\frac{\ell^2}{2k(\ell+1)^2}\cdot\frac{\sum\limits_{0\leq i\leq k-\ell}\binom{n-1}{i}}{\sum\limits_{0\leq i\leq k-\ell-1}\binom{n-1}{i}}\overset{\eqref{ineq-4-1}}{\geq} \frac{(n-1-k)\ell^2}{2k^2(\ell+1)^2}\geq 1 \mbox{ for } n\geq 5k^2,
 \]
 $f(n,k,\ell)$ is decreasing as a function of $\ell$.
 Moreover, by \eqref{ineq-crosshb} we have
 \[
 \sum_{\beta \leq \ell \leq k } \lambda_\ell \leq 1.
 \]
Hence,
 \begin{align}\label{ineq-3-6}
 \sum_{\beta\leq \ell\leq k} |\hi(\hf^{(\ell)},\hg)|\leq f(n,k,\beta) \leq  f(n,k,3) =72k\sum_{0\leq i\leq k-3}\binom{n-1}{i}.
 \end{align}
 Using \eqref{ineq-4-1} and \eqref{ineq-4-2}, for $n\geq 2k+3$ we have
\begin{align*}
\sum_{\beta\leq \ell\leq k} |\hi(\hf^{(\ell)},\hg)|&\leq 72k\sum_{0\leq i\leq k-3}\binom{n-1}{i}\\
  &\overset{\eqref{ineq-4-2}}{\leq}  \frac{72k^2}{n-1-k}\sum_{0\leq i\leq k-2}\binom{n-1}{i}\\
 &\overset{\eqref{ineq-4-1}}{\leq} \frac{72k^2}{n-1-k}\frac{n-1-3}{n-1-3k}\sum_{0\leq i\leq k-2}\binom{n-4}{i}.
 \end{align*}
 Since $n\geq 100k^2\geq 100k$, we infer
\[
\frac{72k^2}{n-1-k} \leq \frac{8}{11} \mbox{ and } \frac{n-1-3}{n-1-3k} \leq \frac{33}{32}.
\]
It follows that for $\beta\geq 3$
\begin{align}\label{ineq-3-10}
\sum_{\beta\leq \ell\leq k} |\hi(\hf^{(\ell)},\hg)| \leq \frac{3}{4}\sum_{0\leq i\leq k-2}\binom{n-4}{i}.
\end{align}
 Using \eqref{ineq-3-5} and \eqref{ineq-3-10}, we have
\begin{align*}
 |\hi(\hf,\hg)|\leq  |\hi(\hf',\hg)|+\sum_{\beta\leq \ell\leq k} |\hi(\hf^{(\ell)},\hg)|\leq 4\sum_{0\leq i\leq k-2}\binom{n-4}{i}<|\hi(\ha_1,\ha_2)|.
 \end{align*}

 {\bf Case 2.} $\beta=2$ and $\alpha>2$.

By  \eqref{ineq-3-10} we have
 \[
 \sum_{3\leq \ell\leq k} |\hi(\hf^{(\ell)},\hg)| \leq \frac{3}{4}\sum_{0\leq i\leq k-2}\binom{n-4}{i}.
 \]
 Since $\alpha>2$, it follows that $\hf^{(2)}$ is a star. By \eqref{ineq-cor3-4} and \eqref{ineq-3-9}, we have
\begin{align*}
 |\hi(\hf^{(2)},\hg)|<\frac{13}{4}\sum_{0\leq i\leq k-2}\binom{n-4}{i}.
\end{align*}
Thus,
\[
|\hi(\hf,\hg)|\leq  |\hi(\hf^{(2)},\hg)|+\sum_{3\leq \ell\leq k} |\hi(\hf^{(\ell)},\hg)|\leq 4\sum_{0\leq i\leq k-2}\binom{n-4}{i}<|\hi(\ha_1,\ha_2)|.
\]

 {\bf Case 3.} $\beta=\alpha=2$.

  Since $\hb_1^{(2)}, \hb_2^{(2)}$ are cross-intersecting, we see that $\nu(\hb_1^{(2)})\leq 2$ and $\nu(\hb_2^{(2)})\leq 2$. Moreover, $\beta=\alpha=2$ implies $\tau(\hb_1^{(2)})\geq 2$ and $\tau(\hb_2^{(2)})\geq  2$. It follows that $\hb_1^{(2)}, \hb_2^{(2)}$ are  either both triangles or both  subgraphs of $K_4$ with a matching of size two.

 {\bf Case 3.1.} $\hb_1^{(2)}, \hb_2^{(2)}$  are both  triangles.

Without loss of generality, assume that $\hb_1^{(2)}=\hb_2^{(2)}=\{(1,2),(1,3),(2,3)\}$. By saturatedness, we have
 \[
 \hf=\hg= \ha_3=\left\{A\in \binom{[n]}{k}\colon |A\cap \{1,2,3\}|\geq 2\right\}.
 \]
 Therefore,
 \begin{align*}\label{ineq-3-8}
 |\hi(\hf,\hg)|&= |\hi(\ha_3,\ha_3)|\\[5pt]
 &=3\sum_{0\leq i\leq k-2}\binom{n-3}{i}+3\sum_{0\leq i\leq k-3}\binom{n-3}{i}+\sum_{0\leq i\leq k-4}\binom{n-3}{i}\\[5pt]
 &\overset{\eqref{ineq-4-1}}{\leq} \frac{n-4}{n-3-k}\left(3\sum_{0\leq i\leq k-2}\binom{n-4}{i}+3\sum_{0\leq i\leq k-3}\binom{n-4}{i}+\sum_{0\leq i\leq k-4}\binom{n-4}{i}\right).
 \end{align*}
Since $n\geq 13k$ implies $\frac{n-4}{n-3-k}\leq \frac{13}{12}$, we obtain that
\begin{align*}
 |\hi(\hf,\hg)| < \frac{13}{4}\sum_{0\leq i\leq k-2}\binom{n-4}{i}+\frac{13}{4}\sum_{0\leq i\leq k-3}\binom{n-4}{i}+2\sum_{0\leq i\leq k-4}\binom{n-4}{i}< |\hi(\ha_1,\ha_2)|.
\end{align*}

{\bf Case 3.2.} $\hb_1^{(2)}, \hb_2^{(2)}$  are both subgraphs of $K_4$ with a matching of size two.

By symmetry, we may assume that $(1,3),(2,4) \in \hb_1^{(2)}$ and $(1,2),(3,4)\in \hb_2^{(2)}$. We further assume that $|\hb_1^{(2)}|\geq  |\hb_2^{(2)}|$.

{\bf Case 3.2.1.} $\hb_1^{(2)}= \{(1,3),(2,4),(1,4),(2,3)\}$ and $\hb_2^{(2)}= \{(1,2),(3,4)\}$.

 Since $\hf,\hg$ are saturated, we have $\hf=\hf^{(2)}$ and $\hg=\hg^{(2)}$. Thus,
 \begin{align*}
 |\hi(\hf,\hg)|&=4\sum_{0\leq i\leq k-2}\binom{n-4}{i}+6\sum_{0\leq i\leq k-3}\binom{n-4}{i} +4\sum_{0\leq i\leq k-4}\binom{n-4}{i}\\[5pt]
 &\qquad\qquad +\sum_{0\leq i\leq k-5}\binom{n-4}{i}\\[5pt]
 &<|\hi(\ha_1,\ha_2)|.
 \end{align*}

{\bf Case 3.2.2.} $\hb_1^{(2)}= \{(1,3),(2,4),(1,4)\}$ and $\hb_2^{(2)}= \{(1,2),(3,4),(1,4)\}$.

Since $\hf,\hg$ are saturated, we have $\hf=\hf^{(2)}$ and $\hg=\hg^{(2)}$. Thus,
\begin{align*}
 |\hi(\hf,\hg)|&=4\sum_{0\leq i\leq k-2}\binom{n-4}{i}+6\sum_{0\leq i\leq k-3}\binom{n-4}{i} +4\sum_{0\leq i\leq k-4}\binom{n-4}{i}\\[5pt]
 &\qquad\qquad +\sum_{0\leq i\leq k-5}\binom{n-4}{i}\\[5pt]
 &<|\hi(\ha_1,\ha_2)|.
\end{align*}

{\bf Case 3.2.3.} $\hb_1^{(2)}= \{(1,3),(2,4)\}$ and $\hb_2^{(2)}= \{(1,2),(3,4)\}$.

By Lemma \ref{lem3-1}, we have $\hs_{13}\cup \hs_{24}\subset \hf$ and $\hs_{12}\cup \hs_{34}\subset \hg$. Let $\hf' =\hf \setminus (\hs_{13}\cup \hs_{24})$ and $\hg' =\hg \setminus (\hs_{12}\cup \hs_{34})$. Since $\hb_1^{(2)}, \hg'$ are cross-intersecting,   $G\cap \{1,3\}\neq \emptyset$ and  $G\cap \{2,4\}\neq \emptyset$ for all $G\in \hg'$. Moreover, $G\notin \hs_{12}\cup \hs_{34}$. It follows that $G\cap [4]=\{1,4\}$ or $G\cap [4]=\{2,3\}$ for all $G\in \hg'$. Similarly, $F\cap [4]=\{1,4\}$ or $F\cap [4]=\{2,3\}$ for all $F\in \hf'$. Let
\[
\hf_{14}'= \{F\colon F\in \hf', F\cap [4]=\{1,4\}\},\  \hf_{23}'= \{F\colon F\in \hf', F\cap [4]=\{2,3\}\}
\]
and
\[
\hg_{14}'= \{G\colon G\in \hg', G\cap [4]=\{1,4\}\},\ \hg_{23}'= \{G\colon G\in \hg', G\cap [4]=\{2,3\}\}.
\]
Since $\hf_{14}',\hg_{23}'$ are cross-intersecting and $\hf_{23}',\hg_{14}'$ are cross-intersecting, by \eqref{ineq-2-3} we have
\begin{align*}
|(\hf_{14}'\wedge\hg_{23}')\cup (\hf_{23}'\wedge\hg_{14}')|\leq 2\sum_{0\leq i\leq k-3}\binom{n-6}{i}+\sum_{0\leq i\leq k-4}\binom{n-6}{i}.
\end{align*}
Note that $\hi(\hf_{14}',\hg\setminus\hg_{23}')\subset \hi(\hs_{13}\cup \hs_{24},\hs_{12}\cup \hs_{34})$ and $\hi(\hf_{23}',\hg\setminus\hg_{14}')\subset \hi(\hs_{13}\cup \hs_{24},\hs_{12}\cup \hs_{34})$. Thus,
\begin{align*}
|\hi(\hf,\hg)| &= |\hi(\hs_{13}\cup \hs_{24},\hs_{12}\cup \hs_{34})|+|(\hf_{14}'\wedge\hg_{23}')\cup (\hf_{23}'\wedge\hg_{14}')|\\[5pt]
&\leq 4\sum_{0\leq i\leq k-2}\binom{n-4}{i}+6\sum_{0\leq i\leq k-3}\binom{n-4}{i} +4\sum_{0\leq i\leq k-4}\binom{n-4}{i}\\[5pt]
&\qquad +\sum_{0\leq i\leq k-5}\binom{n-4}{i}+2\sum_{0\leq i\leq k-3}\binom{n-6}{i}+\sum_{0\leq i\leq k-4}\binom{n-6}{i}\\[5pt]
&=|\hi(\ha_1,\ha_2)|.
\end{align*}

{\bf Case 3.2.4.} $\hb_1^{(2)}= \{(1,3),(2,4),(1,4)\}$ and $\hb_2^{(2)}= \{(1,2),(3,4)\}$.

By Lemma \ref{lem3-1} (ii), we have $\hs_{13}\cup \hs_{24}\cup\hs_{14}\subset \hf$ and $\hs_{12}\cup \hs_{34}\subset \hg$. Let $\hf' =\hf \setminus (\hs_{13}\cup \hs_{24}\cup\hs_{14})$ and $\hg' =\hg \setminus (\hs_{12}\cup \hs_{34})$. Since $\hb_1^{(2)}, \hg'$ are cross-intersecting, $G\cap [4]=\{1,4\}$ for all $G\in \hg'$. Similarly, $F\cap[4]=\{2,3\}$ for all $F\in \hf'$. Since $\hf',\hg'$ are cross-intersecting, by \eqref{ineq-1-1} we have
\[
|\hf_{23}'\wedge\hg_{14}'| \leq  \sum_{0\leq i\leq k-3}\binom{n-5}{i}<2\sum_{0\leq i\leq k-3}\binom{n-6}{i}+\sum_{0\leq i\leq k-4}\binom{n-6}{i}.
\]
Note that $\hi(\hf_{23}',\hg\setminus\hg_{14}')\subset \hi(\hs_{13}\cup \hs_{24},\hs_{12}\cup \hs_{34})$. Thus,
\begin{align*}
|\hi(\hf,\hg)| &= |\hi(\hs_{13}\cup \hs_{24},\hs_{12}\cup \hs_{34})|+|\hf_{23}'\wedge\hg_{14}'|\leq |\hi(\ha_1,\ha_2)|.
\end{align*}
\end{proof}

\section{Distinct intersections in a $t$-intersecting family}

In this section, we determine the maximum number of distinct intersections in a $t$-intersecting family.

Since $\hf\subset \tilde{\hf}$ implies $\hi(\hf)\subset \hi(\tilde{\hf})$, we may always assume that $\hf$ is saturated. Let $\hb=\hb_t(\hf)$ be the family of minimal (for containment) sets in $\hht_t(\hf)$.

\begin{lem}\label{lem4-1}
Suppose that $\hf\subset \binom{[n]}{k}$ is a saturated $t$-intersecting family. Then (i) and (ii) hold.
\begin{itemize}
  \item[(i)] $\hb$ is a $t$-intersecting antichain,
  \item[(ii)] $\hf=\left\{H\in \binom{[n]}{k}\colon \exists B\in \hb, B\subset H\right\}$.
\end{itemize}
\end{lem}

\begin{proof}
 (i) Clearly, $\hb$ is an anti-chain. Suppose for contradiction that $B,B'\in \hb$ but $|B\cap B'|<t$. If $|B|=|B'|=k$, then $B,B'\in \hf$ as $\hf$ is saturated, a contradiction. If $|B'|<k$, then there exists $F'\supset B'$ such that $|F'|=k$ and $|F'\cap B|=|B'\cap B|<t$. By definition $F'\in \hht_t(\hf)$. Since $\hf$ is saturated, we see that $F'\in \hf$. But this contradicts the assumption that $B$ is a $t$-transversal. Since $\hf$ is saturated, (ii) is immediate from the definition of $\hb$.
\end{proof}

Let $r(\hb)=\max\{|B|\colon B\in \hb\}$ and $s(\hb)=\min\{|B|\colon B\in \hb\}$. For any $\ell$ with $s(\hb)\leq \ell \leq r(\hb)$ define
\[
\hb^{(\ell)} = \left\{B\in \hb\colon |B|=\ell\right\} \mbox{ and } \hb^{(\leq \ell)} = \bigcup_{i=s(\hb)}^\ell\hb^{(i)}.
\]
It is easy to see that $s(\hb_t(\hf))=\tau_t(\hf)$.

\begin{lem}\label{lem4-2}
Suppose that $\hf\subset \binom{[n]}{k}$ is a saturated $t$-intersecting family and $\hb=\hb_t(\hf)$. If $s(\hb)\geq t+1$ and $\tau_t(\hb^{(\leq r)})\geq t+1$, then
\begin{align}\label{ineq-thb}
\sum_{r\leq \ell \leq k} \left(\binom{\ell}{t}\ell k^{\ell-t-1}\right)^{-1}|\hb^{(\ell)}|\leq 1.
\end{align}
\end{lem}

\begin{proof}
For the proof we use a branching process. During the proof {\it a sequence} $S=(x_1,x_2,\ldots,x_\ell)$ is an ordered sequence of distinct elements of $[n]$ and we use $\widehat{S}$ to denote the underlying unordered set $\{x_1,x_2,\ldots,x_\ell\}$. At the beginning, we assign weight 1 to the empty sequence $S_{\emptyset}$. At the first stage, we choose $B_1\in \hb_t$ with $|B_1|=s(\hb)\geq t+1$. For any $t$-subset $\{x_1,\ldots,x_t\}\subset B_1$, define one sequence $(x_1,\ldots,x_t)$ and assign the weight $\binom{s(\hb)}{t}^{-1}$ to it.

At the second stage, since $\tau_t(\hb^{(\leq r)})\geq t+1$, for each $t$-sequence $S=(x_1,\ldots,x_t)$ we may choose $B\in \hb^{(\leq r)}$ such that $|\widehat{S}\cap B|<t$. Then we replace $S=(x_1,\ldots,x_t)$ by $|B\setminus \widehat{S}|$ $(t+1)$-sequences of the form $(x_1,\ldots,x_t,y)$ with $y\in B\setminus \widehat{S}$ and weight $\frac{w(S)}{|B\setminus \widehat{S}|}$.

In each subsequent stage, we pick a sequence $S=(x_1,\ldots,x_p)$ and denote its weight by $w(S)$. If $|\widehat{S}\cap B|\geq t$ for all $B\in \hb$ then we do nothing. Otherwise we pick $B\in \hb$ satisfying $|\widehat{S}\cap B|< t$ and replace $S$ by the $|B\setminus \widehat{S}|$ sequences $(x_1,\ldots,x_p,y)$ with $y\in B\setminus \widehat{S}$ and assign weight $\frac{w(S)}{|B\setminus \widehat{S}|}$ to each of them. clearly, the total weight is always 1.

We continue until $|\widehat{S}\cap B|\geq t$ for all sequences and all $B\in \hb$. Since $[n]$ is finite, each sequence has length at most $n$ and  eventually the process stops. Let $\hs$ be the collection of sequences that survived in the end of the branching process and let $\hs^{(\ell)}$ be the collection of sequences in $\hs$ with length $\ell$.

\begin{claim}
To each $B\in \hb^{(\ell)}$ with $\ell\geq r$ there is some sequence $S\in \hs^{(\ell)}$ with $\widehat{S}=B$.
\end{claim}
\begin{proof}
Let us suppose the contrary and let $S=(x_1,\ldots,x_p)$ be a sequence of maximal length that occurred
at some stage of the branching process satisfying $\widehat{S}\subsetneqq B$. Since $\hb$ is $t$-intersecting,
 $|B\cap B_1|\geq t$, implying that $p\geq t$. Since $\widehat{S}$ is a proper subset of $B$, there exists $F\in \hf$ with $|\widehat{S} \cap F|<t$. In view of Lemma \ref{lem4-1} (ii) we can find  $B'\in \hb$ such that $|\widehat{S} \cap B'|<t$. Thus at some point we picked $S$ and some $\tilde{B}\in \hb$ with $|\widehat{S} \cap \tilde{B}|<t$. Since $\hb$ is $t$-intersecting, $|B\cap \tilde{B}|\geq t$. Consequently, for each $y\in B\cap (\tilde{B}\setminus \widehat{S})$ the sequence $(x_1,\ldots,x_p,y)$ occurred in the branching process. This contradicts the maximality of $p$. Hence there is an $S$ at some stage satisfying $\widehat{S}= B$. Since $\hb$ is $t$-intersecting, $|\widehat{S}\cap B'|\geq t$ for all $B'\in \hb$. Thus $\widehat{S}\in \hs$ and the claim holds.
\end{proof}
By Claim 2, we see that $|\hb^{(\ell)}|\leq |\hs^{(\ell)}|$.  Let $S=(x_1,\ldots,x_\ell)\in \hs^{(\ell)}$ and let $S_i=(x_1,\ldots,x_i)$ for $i=1,\ldots,\ell$. At the first stage, $w(S_t)=1/\binom{s(\hb)}{t}$.   Assume that $B_i$ is the selected set when replacing $S_{i-1}$ in the  branching process for $i=t+1,\ldots,\ell$. Clearly,  $x_i\in B_i$, $B_{t+1}\in \hb^{(\leq r)}$ and
\[
w(S)= \frac{1}{\binom{s(\hb)}{t}}\prod_{i=t+1}^\ell \frac{1}{|B_i\setminus \widehat{S_{i-1}}|}.
\]
Note that $s(\hb)\leq r\leq  \ell$, $|B_{t+1}\setminus \widehat{S_t}|\leq \ell$ and $|B_i\setminus \widehat{S_{i-1}}|\leq k$ for $i\geq t+2$. It follows that
\[
w(S)\geq \left(\binom{\ell}{t}\ell k^{\ell-t-1}\right)^{-1}.
\]
Thus,
\[
\sum_{r\leq \ell \leq k} \left(\binom{\ell}{t}\ell k^{\ell-t-1}\right)^{-1}|\hb^{(\ell)}|\leq \sum_{r\leq \ell \leq k} \sum_{S\in \hs^{(\ell)}}w(S) \sum_{S\in \hs}w(S)=1.
\]
\end{proof}

\begin{lem}\label{lem-4-3}
Suppose that $\tau_t(\hb^{(t+1)}\geq t+1$. Then $\hf=\ha(n,k,t)$.
\end{lem}

\begin{proof}
Choose $B_1,B_2 \in \hb^{(t+1)}$ and assume by symmetry that $B_i=[t]\cup \{t+i\}$ for $i=1,2$. Since $\tau_t(\hb^{(t+1)})\geq t+1$, we may choose $B_3\in \hb^{(t+1)}$  satisfying $[t]\nsubseteq B_3$. Now $|B_2\cap B_i|\geq t$ implies $\{t+1,t+2\}\subset B_3$. Using $|B_3|=t+1$, by symmetry we may assume that $B_3=[t+2]\setminus \{t\}$. Now take an arbitrary $F\in \hf$. It is clear that $|F\cap B_i|\geq t$ can only hold for all $1\leq i\leq 3$ if $|F\cap [t+2]|\geq t+1$. That is $\hf\subset \ha(n,k,t)$. Since $\hf$ is saturated, $\hf=\ha(n,k,t)$.
\end{proof}

 \begin{proof}[Proof of Theorem \ref{main-3}]
By $\eqref{eq-1-2}$ and  $\eqref{ineq-4-1}$, we have
 \begin{align*}
 |\hi(\hs_{[t]})|&= \sum_{0\leq i\leq k-t-1}\binom{n-t}{i}\\[5pt]
 &\overset{\eqref{ineq-4-1}}{\leq} \frac{n-t-2}{n-t-2k}\sum_{0\leq i\leq k-t-1}\binom{n-t-2}{i}\\[5pt]
 &< \binom{t+2}{t}\sum_{0\leq i\leq k-t-1}\binom{n-t-2}{i}\\[5pt]
 &<|\hi(\ha(n,k,t))|.
 \end{align*}
Thus, we may assume that $s=s(\hb)\geq t+1$.
 Let us partition $\hf$ into $\hf^{(s)}\cup \ldots\cup\hf^{(k)}$ where $F\in \hf^{(\ell)}$ if $\max\{|B|\colon B\in \hb, B\subset F\}=\ell$. Set
 \[
\hi_{\ell} = \left\{F\cap F'\colon F\in\hf^{(\ell)}, F'\in \hf^{(s)}\cup \ldots\cup\hf^{(\ell)} \right\}.
 \]
 Then
 \[
 |\hi(\hf)| \leq \sum_{s\leq \ell \leq k} |\hi_{\ell}|.
 \]
 The point is that for $F\in \hf^{(\ell)}$ and $B\subset F, B\in \hb^{\ell}$ for an arbitrary $F'\in \hf$,
 \[
 F\cap F' =(B\cap F')\cup ((F\setminus B)\cap F').
 \]
 Note that $|B\cap F'|\geq t$ and $|(F\setminus B)\cap F'| \leq |F\setminus B|=k-\ell$. It follows that for $s\leq \ell\leq k$
 \begin{align}\label{ineq-4-7}
 |\hi_\ell| \leq \left(\sum_{t\leq j\leq \ell}\binom{\ell}{j}\right)|\hb^{(\ell)}|\sum_{0\leq i\leq k-\ell }\binom{n-t}{i}.
 \end{align}

 Let $\alpha$ be the smallest integer such that $\tau_t(\hb^{(\leq \alpha)})\geq t+1$. The family $\hf'=\cup_{i=l}^{\alpha-1}\hf^{(i)}$ is a trivial $t$-intersecting family. By \eqref{ineq-4-1}, we have for $n\geq 5k$
 \begin{align}\label{ineq-4-4}
 \left|\bigcup_{i=s}^{\alpha-1}\hi_i\right|\leq |\hi(\hs_{[t]})|&=\sum_{0\leq i\leq k-t-1}\binom{n-t}{i}\nonumber\\[5pt]
 &\overset{\eqref{ineq-4-1}}{\leq}  \frac{n-t-2}{n-t-2k}\sum_{0\leq i\leq k-t-1}\binom{n-t-2}{i}\nonumber\\[5pt]
 &\leq 2\sum_{0\leq i\leq k-t-1}\binom{n-t-2}{i}.
 \end{align}

 If $\alpha=s= t+1$, then $\hb^{(t+1)}$ is a $t$-intersecting $(t+1)$-uniform family with $t$-covering number $t+1$. By Lemma \ref{lem-4-3}, $\hf=\ha(n,k,t)$ and there is nothing to prove. Thus we may assume that $\alpha\geq t+2$.

Define
\begin{align*}
f(n,k,\ell)= \left(\sum_{t\leq j\leq \ell}\binom{\ell}{j}\right)\binom{\ell}{t}\ell k^{\ell-t-1}\sum_{0\leq i\leq k-\ell}\binom{n-t}{i}
\end{align*}
and let
\[
\lambda_\ell =\left(\binom{\ell}{t}\ell k^{\ell-t-1}\right)^{-1}|\hb^{(\ell)}|.
\]
Then by \eqref{ineq-4-7}
\begin{align}\label{ineq-4-8}
\sum_{\alpha \leq \ell\leq k} |\hi_\ell| = \sum_{\alpha \leq \ell\leq k}\lambda_\ell  \cdot f(n,k,\ell).
\end{align}
By \eqref{ineq-4-2} and  \eqref{ineq-4-6}, we have
 \begin{align*}
 \frac{f(n,k,\ell)}{f(n,k,\ell+1)}&=\frac{\sum\limits_{t\leq j\leq \ell}\binom{\ell}{j}}{\sum\limits_{t\leq j\leq \ell+1}\binom{\ell+1}{j}} \cdot\frac{\binom{\ell}{t}\ell k^{\ell-t-1}}{ \binom{\ell+1}{t}(\ell+1) k^{\ell-t}}\cdot\frac{\sum\limits_{0\leq i\leq k-\ell}\binom{n-t}{i}}{\sum\limits_{0\leq i\leq k-\ell-1}\binom{n-t}{i}}\\[5pt]
 &\geq \frac{1}{2(t+1)} \cdot \frac{(\ell+1-t)\ell}{(\ell+1)^2k}\cdot \frac{n-t-k}{k}.
 \end{align*}
 By $\ell\geq t+1\geq 3$, we have
 \[
 \frac{\ell+1-t}{\ell+1} \cdot \frac{\ell}{\ell+1}\geq \frac{2}{t+2}\cdot \frac{3}{4}\geq \frac{3}{2(t+2)}.
 \]
 Then by $n\geq \frac{4}{3}(t+2)^2k^2$
\[
\frac{f(n,k,\ell)}{f(n,k,\ell+1)}\geq \frac{3(n-t-k)}{4(t+1)(t+2)k^2}\geq 1.
\]
Hence $f(n,k,\ell)$ is  decreasing as a function of $\ell$.
Moreover, \eqref{ineq-thb} implies $\sum\limits_{\alpha \leq \ell\leq k} \lambda_{\ell} \leq 1$. From \eqref{ineq-4-8} we see
\[
\sum_{\alpha \leq \ell\leq k} |\hi_\ell| \leq f(n,k,\alpha)\leq f(n,k,t+2).
\]
Therefore,
 \begin{align*}
  \sum_{\alpha\leq \ell\leq k} |\hi_\ell| &\leq \left(\binom{t+2}{t}+\binom{t+2}{t+1}+\binom{t+2}{t+2}\right)\binom{t+2}{t}(t+2) k\sum_{0\leq i\leq k-t-2}\binom{n-t}{i}\nonumber\\[5pt]
  &\overset{\eqref{ineq-4-2}}{\leq} \frac{(t+2)^2(t+1)(t^2+5t+8)k}{4}\cdot \frac{k}{n-t-k}\sum_{0\leq i\leq k-t-1}\binom{n-t}{i}\nonumber\\[5pt]
  &\overset{\eqref{ineq-4-1}}{\leq} \frac{(t+2)^2(t+1)(t+2)(t+4)k^2(n-t-2)}{4(n-t-k)(n-t-2k)}\sum_{0\leq i\leq k-t-1}\binom{n-t-2}{i}.
 \end{align*}
 Note that $n\geq 5k$ implies
\[
\frac{n-t-2}{n-t-2k} \leq 2
\]
 and  $n\geq 3(t+2)^3k^2$, $t\geq 2$ imply
 \[
 \frac{(t+2)^3(t+1)(t+4)k^2}{4(n-t-k)} \leq \frac{1}{2}\left(\binom{t+2}{2}-2\right).
 \]
 It follows that
 \begin{align}\label{ineq-4-5}
  \sum_{\alpha\leq \ell\leq k} |\hi_\ell| \leq \left(\binom{t+2}{2}-2\right)\sum_{0\leq i\leq k-t-1}\binom{n-t-2}{i}.
 \end{align}
 By \eqref{ineq-4-4} and \eqref{ineq-4-5}, we obtain that
\begin{align*}
|\hi(\hf)|\leq  \left|\bigcup_{i=s}^{\alpha-1}\hi_i\right|+\sum_{\alpha \leq \ell\leq k}|\hi_\ell| < |\hi(\ha(n,k,t))|,
\end{align*}
concluding the proof of the theorem.
 \end{proof}

 \section{Further problems and results}

 In their seminal paper \cite{ekr} Erd\H{o}s, Ko and Rado actually proved their main result for antichains. Namely, instead of considering $k$-graphs $\hf\subset\binom{[n]}{k}$ they suppose that $\hf$ is an antichain of rank $k$, that is $|F|\leq k$ for all $F\in \hf$. The reason that this tendency has all but disappeared from recent research is that a $t$-intersecting antichain $\hf$ of rank $k$ which is not $k$-uniform can always be replaced by a t-intersecting family $\tilde{\hf}\subset \binom{[n]}{k}$ with $|\tilde{\hf}|>|\hf|$. The way to do is to apply an operation on antichains discovered already  by Sperner \cite{Sperner}.

 For a family $\ha\subset \binom{[n]}{a}$ define its {\it shade} $\sigma^+(\ha)$ by
 \[
 \sigma^+(\ha) = \left\{B\in \binom{[n]}{a+1}\colon \exists A\in \ha, A\subset B\right\}.
 \]
Sperner \cite{Sperner} proved that for $a<n/2$, $|\sigma^+(\ha)|\geq |\ha|$ with strict inequality unless $a=\frac{n-1}{2}$ and $\ha=\binom{[n]}{\frac{n-1}{2}}$. Let $\hf\subset 2^{[n]}$ be a $t$-intersecting antichain of rank $k$, $n\geq 2k-t$. Suppose that $a=\min\{|F|\colon F\in \hf\}$ and $a<k$. Define
\[
\hf^{(a)}=\{F\in \hf\colon |\hf|=a\} \mbox{ and } \tilde{\hf}=(F\setminus \hf^{(a)})\cup \sigma^+(\hf^{(a)}).
\]
Then not only is $\tilde{\hf}$ a $t$-intersecting antichain of rank $k$ with $|\tilde{\hf}|>|\hf|$ but $\hi(\tilde{\hf})\supset \hi(\hf)$ can be checked easily as well. This shows that it was reasonable to restrict our attention to $k$-uniform families.

However there is a related, very natural problem.

\begin{prob}
Determine or estimate $\max |\hi(\ha)|$ over all antichain $\ha \subset 2^{[n]}$.
\end{prob}

\begin{example}
Let $\ell\leq \frac{n}{2}$ and define $\ha=\binom{[n]}{n-\ell}$. Clearly,
\[
\hi(\ha) =\left\{B\subset [n]\colon n-2\ell\leq |B|<n-\ell\right\}.
\]
Choosing $\ell=\lfloor n/3\rfloor$, we have
\[
|\hi(\ha)| =2^n-\sum_{0\leq i\leq \lfloor\frac{n}{3}\rfloor} \binom{n}{i}-\sum_{0\leq j\leq n-2\lfloor \frac{n}{3}\rfloor}\binom{n}{j}.
\]
\end{example}

\begin{prop}
If $\ha \subset 2^{[n]}$ is an antichain, then $|\hi(\ha)|<2^n-\sqrt{2}^n$.
\end{prop}
\begin{proof}
Note that
\[
|\hi(\ha)|\leq \binom{|\ha|}{2}.
\]
Consequently, if $|\ha|\leq \sqrt{2}^n$ then $|\hi(\ha)|<2^n/2<2^n-\sqrt{2}^n$. Thus we can assume $|\ha|> \sqrt{2}^n$. Since $\ha\cap\hi(\ha)=\emptyset$, we have
\[
|\hi(\ha)|\leq 2^n-|\ha|<2^n-\sqrt{2}^n.
\]
\end{proof}

Two families $\ha,\hb$ are called {\it cross-Sperner} if $A\not\subset B$ and $B\not\subset A$ hold for all $A\in \ha$, $B\in \hb$. Set
\[
\hi(\ha,\hb) =\{A\cap B\colon A\in \ha,B\in \hb\}.
\]
Define
\[
m(n) =\max\{|\hi(\ha,\hb)|\colon \ha,\hb\subset 2^{[n]} \mbox{ are cross-Sperner}\}.
\]
\begin{example}
Let $[n]=X\cup Y$ be a partition. Define
\[
\ha=\{A\cup Y\colon A\subsetneq X\},\ \hb=\{X\cup B\colon B\subsetneq Y\}.
\]
Then
\[
\hi(\ha,\hb) =\left\{A\cup B\colon A\subsetneq X, B\subsetneq Y\right\}
\]
and
\[
|\hi(\ha,\hb)|=2^n-2^{|X|}-2^{|Y|}+1.
\]
\end{example}

\begin{thm}
$m(n)= 2^n-2\cdot2^{n/2}+1$ holds for $n=2d$ even.
\end{thm}
\begin{proof}
The lower bound comes from the example with $|X|=\lfloor n/2\rfloor$, $|Y|=\lceil n/2\rceil$. Note that
for $A,A'\in\ha$, $B,B'\in \hb$ the cross-Sperner property implies $A\not\subset A'\cap B'$, $B\not\subset A'\cap B'$. In particular,
\[
\ha\cap \hi(\ha,\hb)=\emptyset=\hb\cap \hi(\ha,\hb).
\]
Cross-Sperner property implies $\ha\cap \hb=\emptyset$ and $[n]\notin \ha\cup\hb\cup \hi(\ha,\hb)$. Thus
\[
|\ha|+|\hb|+|\hi(\ha,\hb)|\leq 2^n-1
\]
or equivalently
\begin{align}\label{ineq-crossSperner1}
|\hi(\ha,\hb)|\leq 2^n-|\ha|-|\hb|-1.
\end{align}
Obviously,
\begin{align}\label{ineq-crossSperner2}
|\hi(\ha,\hb)|\leq |\ha|\cdot|\hb|.
\end{align}
Suppose that $n=2d$ (even).  If $|A|+|B|\geq 2(2^d-1)$, then \eqref{ineq-crossSperner1} implies
\[
|\hi(\ha,\hb)|\leq 2^n-2\cdot 2^d+1.
\]
If $\frac{|A|+|B|}{2}\leq 2^d-1$ then the inequality between arithmetic and geometric mean yields via \eqref{ineq-crossSperner2}:
\[
|\hi(\ha,\hb)|\leq \left(2^d-1\right)^2=2^n-2\cdot2^d+1.
\]
\end{proof}
However the proof only gives $\hi(\ha,\hb)\leq 2^n-2\cdot2^{n/2}+1$ for $n=2d+1$.
\begin{prob}
For $n=2d+1$, does $m(n)=2^n-2^{d+1}-2^d+1$ hold?
\end{prob}
 

\begin{thebibliography}{10}
 \bibitem{ekr} P. Erd\H{o}s, C. Ko, R. Rado, Intersection theorems for systems of finite sets, Quart. J. Math. Oxford Ser. 12 (1961), 313--320.
 \bibitem{F78} P. Frankl, The Erd\H{o}s-Ko-Rado theorem is true for $n = ckt$,  Coll. Math. Soc. J. Bolyai 18 (1978), 365--375.
\bibitem{F87}P. Frankl, The shifting technique in extremal set theory, Surveys in Combinatorics  123 (1987), 81--110.
\bibitem{F13} P. Frankl, Improved bounds for Erd\H{o}s' matching conjecture, J. Combin. Theory, Ser. A 120 (2013), 1068--1072.
 \bibitem{FKK2022}
 P. Frankl, A. Kupavskii, S. Kiselev, On the maximum number of distinct intersections in an intersecting family, Discrete Math. 345 (2022), 112757. https://doi.org/10.1016/j.disc.2021.112757.
 \bibitem{HM67}
 A.J.W. Hilton, E.C. Milner, Some intersection theorems for systems of finite sets, Q. J. Math. 18 (1) (1967), 369--384.
 \bibitem{Pyber86}
 L. Pyber, A new generalization of the Erd\H{o}s-Ko-Rado theorem, J. Combin. Theory Ser.
A 43 (1986), 85--90.
\bibitem{Sperner}
E. Sperner, Ein Satz \"{u}ber Untermengen einer endlichen Menger, Math. Zeitschrift 27 (1928), 544--548.

 \end{thebibliography}
\end{document}